\documentclass[bezier,11pt]{article}

\usepackage{amssymb}
\usepackage{latexsym}
\usepackage{tikz,epsf}

\usepackage{amsfonts}
\usepackage{amscd}
\usepackage{amsthm}
\usepackage{amsmath}
\usepackage{graphicx}
\usepackage{color}
\usepackage{caption,subcaption}

\newtheorem{theorem}{Theorem}

\newtheorem{prop}[theorem]{Proposition}
\newtheorem{lemma}[theorem]{Lemma}

\newtheorem{cor}[theorem]{Corollary}

\newcommand{\vertex}{\node[vertex]}
\tikzstyle{vertex}=[circle, draw, inner sep=0pt, minimum size=6pt]

\newcommand{\smallqed}{{\tiny ($\Box$)}}

\newcommand{\QEDmark}{\mbox{\textsc{qed}}}
\newcommand{\proofStarter}[1]{\textsc{#1} }

\newcommand{\gid}{\gamma^{\rm ID}}

\begin{document}
\title{On minimum identifying codes in some \\ Cartesian product graphs}
\author{Douglas F. Rall\\
Furman University\\
Greenville, SC, USA\\
doug.rall@furman.edu
 \and Kirsti Wash\\
 Trinity College\\
 Hartford, CT, USA\\
 kirsti.wash@trincoll.edu}
\date{\today}

\maketitle
\begin{abstract}
An identifying code in a graph is a dominating set that also has the property that the closed
neighborhood of each vertex in the graph has a distinct intersection with the set.  The
minimum cardinality of an identifying code, or ID code, in a graph $G$ is called the ID code number of $G$ and is denoted $\gid(G)$. In this paper, we give upper and lower bounds for the ID code number of the prism of a graph, or $G\Box K_2$. In particular, we show that $\gid(G \Box K_2) \ge \gid(G)$ and we show that this bound is sharp. We also give upper and lower bounds for the ID code number of grid graphs and a general upper bound for $\gid(G\Box K_2)$.
\end{abstract}


\section{Introduction}\label{sec:intro}
An identifying code, or ID code, in a graph is a dominating set that also has the property
that the closed neighborhood of each vertex in the graph has a distinct intersection with the
set. Thus every vertex of the graph can be uniquely located by using this intersection. Analogous to the domination number, 
the ID code number of a graph $G$ is the minimum cardinality of an ID code of $G$ and is denoted $\gid(G)$. ID codes were first 
introduced in 1998 by Karpovsky, Chakrabarty and Levitin \cite{IDCintro} who used them to analyze fault-detection problems in 
multi-processor systems. Since 1998 ID codes have been studied in many classes of graphs and an excellent, detailed list of 
references on ID codes can be found on Antoine Lobstein's webpage \cite{Lobweb}. 

We shall focus on ID codes in a specific graph product, the Cartesian product.
The {\it Cartesian product} of graphs $G$ and $H$, denoted $G \Box H$, is the graph whose vertex set is
$V(G) \times V(H)$.  Two vertices $(u_1, u_2)$ and $(v_1, v_2)$ in $G \Box H$ are adjacent if either $u_1v_1 \in E(G)$ and $u_2=v_2$,
or $u_1=v_1$ and $u_2v_2 \in E(H)$. When $H = K_2$, we refer to $G \Box K_2$ as the {\it prism of} $G$.
Cartesian products have been studied for some time, and extensive information on their
structural properties can be found in \cite{imklra-08} and \cite{hik-2011}.

With respect to graph products, ID codes have been studied in the direct product of cliques~\cite{rw2012}, 
hypercubes~\cite{bhl2000, hl2002,idcodehyper,kcla1999,m2006}, and infinite grids~\cite{bl2005,chmglpz1999,idcodegrids}. As we will 
be focusing on Cartesian products, some of the more recent results regarding ID codes have been in the study of the Cartesian product 
of cliques~\cite{gms2006, GW2013}, and the Cartesian product of a path and a clique~\cite{J.Hedet}. In light of these results,
 we first focus on the prism of a graph. When studying any parameter in a Cartesian product, an important question is whether
 there exists some formula relating the value of the parameter in the product to the value of the parameter in the underlying factor graphs.
In~\cite{hr2004} the authors prove the following result that relates the domination number of the prism of a graph $G$
to the domination number of $G$.

\begin{theorem}[\cite{hr2004}] \label{gamma}
If $G$ is any graph, then $\gamma(G) \le \gamma(G \Box K_2) \le 2\gamma(G)$.
\end{theorem}
\noindent  Since identifying codes are in the first place dominating sets, it seems natural to suspect that if $G$ has an identifying code then a similar
relationship would hold between $\gid(G)$ and $\gid(G \Box K_2)$.   Namely, it would be natural to suspect
that $\gid(G) \le \gid(G \Box K_2) \le 2\gid(G)$.  Indeed, we will prove that the lower bound in this inequality is correct
and will show that the upper bound need not be true unless we make some additional assumptions on the minimum ID codes of $G$.
It is known that for any graph $G$ of order $n$, $\gid(G) \le n-1$. In \cite{extremalID} Foucaud et al. identify the class of all graphs which 
attain this bound, and interestingly enough, a subset of this class achieves the lower bound $\gid(G \Box K_2) = \gid(G)$.  We also demonstrate 
an infinite family of graphs
with identifying codes that show the  upper bound is sharp.

 Finally, we concentrate on the ID code number of grid graphs, i.e. the Cartesian product of two paths. The problem of finding the exact 
 value for the domination number of grid graphs was quite difficult and finally settled in \cite{grids}. We expect finding the exact value 
 for the ID code number of grid graphs to be just as difficult. In this paper, we give both upper and lower bounds for the ID code number 
 of grid graphs, and we also give a general upper bound for the ID code number of the Cartesian product of a graph $G$ and a path.

The remainder of the paper is organized as follows.  In Section~\ref{sec:defns} we give some useful
definitions and terminology as well as prove some basic facts about minimum ID codes.  In Section~\ref{upper} we prove the natural upper bound
 for the ID code number of the prism of a graph $G$ when an additional assumption is imposed on $G$ and show this bound is sharp.  Section~\ref{lower} 
 is devoted to giving a lower bound for $\gid(G\Box K_2)$ for any {graph $G$. We also prove that the bound is sharp in this section.  In 
 Section~\ref{grids}, we give upper and lower bounds for $\gid(P_m \Box P_n)$ for any positive integers $2 \le m \le n$ and we give a general 
 upper bound for $\gid(G \Box P_m)$.

\section{Definitions and Preliminary Results} \label{sec:defns}

Given a simple undirected graph $G$ and a vertex $x$ of $G$, we let $N(x)$ denote the
\textit{open neighborhood} of $x$, that is, the set of vertices adjacent to $x$.
The \textit{closed neighborhood} of $x$ is  $N[x]=N(x) \cup \{x\}$.  By a \textit{code}
in $G$ we mean any nonempty subset of vertices in $G$.  The vertices in a code are
called \textit{codewords}.  A code $D$ in $G$
is a \textit{dominating set} of $G$ if $D$ has a nonempty intersection with the closed
neighborhood of every vertex of $G$.  The \textit{domination number} of $G$ is the
cardinality of a smallest dominating set of $G$; it is denoted by $\gamma(G)$. A code having
the property that the distance between any two codewords is at least 3 is called a \textit{2-packing}
of $G$, and $\rho_2(G)$ is the smallest cardinality of a 2-packing in $G$.
 For compact writing we denote $N[x] \cap D$ by $I_D(x)$.  A code $D$  \textit{separates}
two distinct vertices $x$ and $y$ if $I_D(x) \ne I_D(y)$. When $D=\{u\}$ we
say that $u$ separates $x$ and $y$.  As mentioned above, an \textit{identifying code}
(\textit{ID code} for short) of $G$ is a code $C$ that is a dominating set of
$G$ with the additional property that $C$ separates every pair of distinct vertices of $G$.
The minimum cardinality of an ID code of $G$ is denoted $\gid(G)$.  Note that any
graph having two vertices with the same closed neighborhood (so-called \textit{twins})
does not have an ID code.  If a graph has no twins, then we say it is \textit{twin-free}.

If $h\in V(H)$, then the subgraph of $G \Box H$ induced by $V(G)\times \{h\}$ is called a \textit{$G$-fiber}
and is denoted by $G^h$.  In the special case of the prism of $G$ we will assume that $\{1,2\}$ is the
vertex set of $K_2$, and these two $G$-fibers are then $G^1$ and $G^2$.  When dealing with the prism we will simplify
the notation and denote the vertex $(g,i)$ by $g^i$ for $i\in [2]$.  Here $[n]$ denotes the set of positive integers
less than or equal to $n$.  The map $p_G: V(G \Box H) \to V(G)$ defined  by $p_G(a,b)=a$ is the \textit{projection onto $G$}.

While our main emphasis is on minimum ID codes in prisms of graphs, we will also need some basic facts about
ID codes in more general Cartesian products.  The proof of the following is straightforward and is omitted.
\begin{prop} If $G$ and $H$ both have minimum degree at least 1,  then $G \Box H$ is twin-free.
\end{prop}

If $C$ is any ID code in a twin-free graph $G$ of order $n$, then $\left\{I_C(x)\right\}_{x \in V(G)}$
is a collection of $n$, pairwise distinct, nonempty subsets of $C$. This fact immediately implies the following result, which was first given in~\cite{IDCintro}.

\begin{prop}[\cite{IDCintro}] \label{prop:upper} Let $G$ be any twin-free graph of order $n$. If $\gid(G) = k$, then $n \le 2^k - 1$.
Equivalently, $\gid(G) \ge \left\lceil \log_2(n+1)\right\rceil$.
\end{prop}

In particular, an easy application of Proposition~\ref{prop:upper} to prisms yields the following corollary.

\begin{cor} \label{cor:upper} If $H$ is any graph of order $m$ with no isolated vertices, then
\[\gid(H \Box K_2) \ge \left\lceil \log_2(2m+1)\right\rceil.\]
\end{cor}

It also follows directly from Corollary~\ref{cor:upper} that if the prism of a graph $G$ has ID code number $3$,
then $G$ has order at most $3$. Thus, we have the following result.

\begin{cor} If $G$ is a twin-free graph with no isolated vertices such that $\gid(G)=3$, then $\gid(G \Box K_2) >3$.
\end{cor}

By more closely analyzing how an identifying code separates vertices in a prism we can deduce some restrictions on ID codes in prisms.

\begin{lemma}\label{lem:upperorder} Let $G$ be a nontrivial, connected graph of order $n$. If $C$ is an identifying code of
$G \Box K_2$ that has $m_i$ codewords in the $G$-layer $G^i$, for $i \in [2]$, then
\[n \le \min\{2^{m_1} - 1 + m_2, 2^{m_2} - 1 + m_1\}.\]
\end{lemma}
\begin{proof}
Let $C$ be any ID code of $G \Box K_2$ and for $i \in [2]$ and let $m_i= |C_i|$ where $C_i = C \cap V(G^i)$.
Note that $\{a^1 : a^2 \not\in C_2\}$, $\{a^1 : a^2 \in C_2\}$ is a partition of $V(G^1)$. Any two vertices in
the former subset are separated by $C_1$, and it follows that $|\{a^1 : a^2 \not\in C_2\}| \le 2^{m_1} - 1$. Clearly
the second of these parts of the partition has cardinality $m_2$. Combining these we get that $n = |V(G^1)| \le 2^{m_1} - 1 + m_2$.
The result follows by applying a similar argument to $G^2$.
\end{proof}

 \vskip5mm
\begin{prop}\label{thm:dom} If the graph $G$ has no isolated vertices, then $\gid(G\Box K_2) > \gamma(G)$.
\end{prop}
\begin{proof}

Suppose to the contrary that $G\Box K_2$ has a minimum ID code $C$ such that $|C| \le \gamma(G)$.
Since $C$ dominates $G\Box K_2$, it follows from~\cite{hr2004} that $\gamma(G)\ge |C|\ge \gamma(G\Box K_2) \ge \gamma(G)$,
and hence $|C| = \gamma(G)$.  As shown in~\cite{hr2004}, it follows that
$C = (D_1 \times \{1\}) \cup (D_2 \times \{2\})$ where $D = D_1 \cup D_2$ and $D$ is a minimum dominating set of $G$
such that $V(G)-N[D_1]=D_2$ and $V(G)-N[D_2]=D_1$.  Let $X = V(G) - D$. Every vertex of $X$ has exactly one neighbor in $D_1$ and
exactly one neighbor in $D_2$. Let $x \in X$ and suppose $\{d\} = N(x) \cap D_1$. It now follows that
$I_C(d,2)= \{(d,1)\} = I_C(x,1)$, which contradicts the assumption that $C$ is an ID code for $G \Box K_2$.
\end{proof}

\section{Upper Bound}\label{upper}

In this section we prove that under a certain condition on the minimum ID codes of a graph the natural
upper bound holds for the ID code number of its prism.

\begin{theorem}\label{thm:twicegamma} If $G$ has a minimum ID code $I$ such that $G[I]$ has no isolated vertices, then $\gid(G\Box K_2) \le 2\gid(G)$.
\end{theorem}

\begin{proof}
Let $D = I \times \{1,2\}$, let $D_1 = I \times \{1\}$, and let $D_2 = I \times \{2\}$. It is clear that $D$ dominates $G \Box K_2$ since $I$ 
dominates $G$. Let $x$ and $y$ be distinct vertices of $G \Box K_2$. We show that $D$ separates $x$ and $y$.
Suppose first that at least one of $x$ and $y$ belongs to $D$. Without loss of generality we assume that $x \in D_1$. If $y \in D_1$, then $x$ 
and $y$ have distinct neighbors in $D_2$. If $y \in G^1 - D_1$, then $x$ has a neighbor in $D_2$ but $y$ does not. If $y \in G^2 - D_2$, then 
since $G[I]$ has no isolated vertices it follows that $x$ has a neighbor in $D_1$, but $y$ does not. Finally, suppose that $y \in D_2$. If 
$p_G(x) = p_G(y)$, then $(N[x]\cap D_1) - N[y] \ne \emptyset$ since $G[I]$ has no isolated vertices. If $p_G(x) \ne p_G(y)$, then $x \in N[x] - N[y]$. 
Thus $D$ separates $x$ and $y$ if at least one of them belongs to $D$. Now suppose that $x \in V(G^1) - D_1$. If $y$ also belongs to $V(G^1) - D_1$, 
then $D$ separates $x$ and $y$ because $I$ separates $p_G(x)$ and $p_g(y)$. On the other hand, if $y \in V(G^2) - D_2$, then 
$N[y] \cap D \subseteq D_2$ while $N[x] \cap D \subseteq D_1$ and thus $D$ separates $x$ and $y$.
\end{proof}

If we do not require that the subgraph of $G$ induced by a minimum ID code has no isolated vertices, then the conclusion may not hold.
As an example, let $X = \{1, 2, 3, 4\}$ and let $Y = \{A : A \subset \{1,2,3,4\} \ \text{and} \ |A|\ge 2\}$. Construct a  bipartite graph $G$ where
$V(G) = X \cup Y$.  In $G$ the vertex $j \in X$ is adjacent to the vertex $A \in Y$ exactly when $j \in A$.
It is clear that $X$ is an identifying code in $G$ and it then follows by Proposition~\ref{prop:upper}
that $\gid(G) \ge \log_2(|V(G)|+1) = 4$.  It can be easily verified that $\gid(G \Box K_2) = 9$, which shows that the conclusion of
Theorem~\ref{thm:twicegamma} does not hold for this graph.

The upper bound given in Theorem~\ref{thm:twicegamma} is sharp.  To see this we consider the infinite class of so-called corona graphs.
For a given graph $H$ the {\it corona} of $H$ is the graph constructed from $H$ by adding a single (new) vertex of
degree $1$ adjacent to each vertex of $H$.  The corona of $H$ is denoted by $H \circ K_1$. Suppose that $H$ is twin-free and connected.  The set of vertices in the original
graph $H$ is a minimum dominating set of $H \circ K_1$ and also separates all pairs of vertices in this corona since $H$ is twin-free.
Consequently, $\gid(H \circ K_1)=|V(H)|$.  As the following proposition shows, we can also determine the identifying code number of the prisms of
a more general class of graphs that includes these coronas.  This result will also then yield an infinite family of graphs that achieve the upper bound given
in Theorem~\ref{thm:twicegamma}.

Let $n$ be any positive integer larger than 1.  The class of graphs $\mathcal{H}_n$  consists of all the finite graphs that can be
obtained from any connected graph of order $n$ by adding at least one new vertex of degree 1 adjacent to each of these $n$ vertices. (Note that
$\mathcal{H}_n$  contains the corona of each connected graph of order $n$.)

\begin{prop}\label{prop:gencorona} If $H \in \mathcal{H}_n$, then $\gid(H\Box K_2) = |V(H)|$.
\end{prop}

\begin{proof} Suppose $H \in \mathcal{H}$, let $u_1, \dots, u_n$ represent the vertices of the underlying graph  of order $n$, and
for each $i \in [n]$ let $x_{i,1}, \dots, x_{i,k_i}$ represent the vertices of degree 1 adjacent to $u_i$. One can easily verify
that $V(H^1)$ is an ID code for $H \Box K_2$. Hence $\gid(H\Box K_2) \le |V(H)|$. Suppose that $C$ is an ID code for $H\Box K_2$.
For each $i \in [n]$, let
\[A_i = \left( \bigcup_{j=1}^{k_i} \{(x_{i,j},1), (x_{i,j},2)\}\right) \cup \{(u_i,1),(u_i,2)\}.\]
We claim that $|A_i \cap C| \ge k_i+1$ for each $i \in [n]$. Note first that if $\{(x_{i,j},1), (x_{i,j}, 2)\} \cap C = \emptyset$
for some $1 \le j \le k_i$, then $\{(u_i, 1), (u_i, 2)\} \subseteq C$ since $C$ dominates $H \Box K_2$. If $k_i=1$, then we are done.
So assume that $k_i >1$. If there exists $\ell \ne j$ such that
\[\{(x_{i,j},1), (x_{i,j},2),(x_{i,\ell},1),(x_{i,\ell},2)\} \cap C = \emptyset,\]
then $C$ does not separate $(x_{i,j},1)$ and $(x_{i,\ell},1)$. So in this case, $|A_i \cap C| \ge k_i +1$.

Next, suppose $|\{(x_{i,j},1), (x_{i,j},2)\} \cap C| \ge 1$ for each $1 \le j \le k_i$. If some $j$ satisfies
$\{(x_{i,j},1),(x_{i,j},2)\} \subseteq C$, then we are done. So we may assume $|\{(x_{i,j},1),(x_{i,j},2)\}\cap C| =1$ for each
 $1 \le j\le k_i$. However, in this case one of $(u_i,1)$ or $(u_i,2)$ is in $C$ for otherwise $(x_{i,j},1)$ and $(x_{i,j},2)$
 are not separated. Thus, $|A_i \cap C| \ge k_i+1$ in each case. This shows that $|C| \ge \sum_{i=1}^n |A_i \cap C| \ge \sum_{i=1}^n (k_i + 1) = |V(H)|$.
\end{proof}

If $H$ is connected and twin-free, then by Proposition~\ref{prop:gencorona} we see that the corona $H \circ K_1$ is a graph that achieves the upper bound in Theorem~\ref{thm:twicegamma}. Hence this bound is achieved for infinitely many graphs.

\section{Lower Bound}\label{lower}

As mentioned in Section~\ref{sec:intro}, Hartnell and Rall show in~\cite{hr2004} that $\gamma(G\Box K_2) \ge \gamma(G)$ and we would naturally 
expect that $\gid(G\Box K_2) \ge \gid(G)$ to be true as well. However, the same projection argument that was used in~\cite{hr2004} creates 
complications when applied to an ID code. In particular, given an ID code $C$ of $G\Box K_2$, $p_G[C]$ need not be an ID code of 
$G$ since $p_G[C]$ may induce isolated edges. However, we show in the following result that we can construct an ID code of $G$ from $p_G[C]$.

\begin{theorem}\label{thm:lowerbound} For any twin-free graph $G$, $\gid(G\Box H) \ge \gid(G)\rho_2(H)$.
\end{theorem}

\begin{proof}
Let $C$ be a minimum ID code of $G\Box H$ and fix a vertex $h \in V(H)$. Let $C' = p_G[C \cap V(G^h)]$. If $C'$ is an ID code of $G$, then we are 
done. So assume there exists at least one pair of vertices $x,y \in V(G)$ such that
$I_{C'}(x) = I_{C'}(y)$. For any pair $x,y$ where $I_{C'}(x) = I_{C'}(y)$, we shall say that $x$ and $y$ are {\it restricted twins with respect 
to $C'$}. Note that $x \sim y$ if $x$ and $y$ are restricted twins is an equivalence relation. It is clear that $x \sim x$ and if $x \sim y$ and 
$y \sim z$ then $I_{C'}(x) = I_{C'}(y) = I_{C'}(z)$. Thus, $y \sim x$ and $x \sim z$. We let $R(x)$ represent the equivalence class of $x$, 
i.e. the set of restricted twins of $x$. It follows that $R(x) \cap R(y) = \emptyset$ or $R(x) = R(y)$ for all $x, y \in V(G)$.

Let $R(a_1), \dots, R(a_m)$  be a complete set of distinct equivalence classes of $\sim$ restricted to $N[C']$, and 
let $R(a_0) = V(G) - N[C']$. Note that $N[C'] = \cup_{i=1}^m R(a_i)$. If $R(a_0) \ne \emptyset$, then we assume $a_0 \in R(a_0)$. Furthermore, 
we may assume that there exists some $m_1 \in [m]$ and we can reindex the $a_i$s if necessary so that $|R(a_i)| >1$ for each $i \in [m_1]$ and 
$|R(a_i)| = 1$ for all $i > m_1$.

{\underline{Claim 1}} If $R(a_0) \ne \emptyset$, then we can choose a set of $|R(a_0)|$ vertices from $V(G)-C'$ that dominates and separates 
each pair of vertices of $R(a_0)$.
\vskip2mm
{\underline{Proof}} We proceed by induction on the cardinality of $R(a_0)$. Suppose first that $R(a_0) = \{a_0\}$. It is clear that $\{a_0\}$ 
dominates and separates $R(a_0)$. Next, assume that $R(a_0) = \{a_0,v\}$. If $a_0$ is not adjacent to $v$, then $\{a_0,v\}$ dominates and
separates $R(a_0)$. So assume that $a_0$ is adjacent to $v$. If $(V(G) - C') \cap N[a_0] = (V(G) - C') \cap N[v]$, then $a_0$ and $v$ are twins in 
$G$ since $I_{C'}(a_0) = I_{C'}(v)$. So either there exists $w \in (V(G) - C') \cap (N[a_0]- N[v])$ or there exists $w \in
 (V(G) - C') \cap (N[v]- N[a_0])$. In either case, $\{a_0, w\}$ separates $R(a_0)$.

 Assume that when $|R(a_0)| = k$, we can choose a set of $k$ vertices to dominate and separate each pair of vertices of $R(a_0)$. Suppose that 
 $|R(a_0)| = \{u_1, \dots, u_k, u_{k+1}\}$. By the inductive hypothesis, there exists a set $W \subseteq V(G)- C'$ that dominates and separates 
 each pair of vertices of $R(a_0) - \{u_{k+1}\}$ and $|W|=k$. If $W$ dominates and separates each pair of vertices in $R(a_0)$, then we are done. 
 So first assume that $W$ does not dominate $u_{k+1}$. Note that since $W$ dominates $R(a_0) - \{u_{k+1}\}$, then $W \cap N[u_{k+1}] \ne W \cap N[u_j]$ 
 for all $1 \le j \le k$. Thus, $W' = W \cup \{u_{k+1}\}$ is a set of $k+1$ vertices that both dominates and separates each pair of vertices of $R(a_0)$.

 Next, suppose that $W$ dominates $u_{k+1}$ but there exists some $j \in [k]$ such that $W \cap N[u_{k+1}] = W \cap N[u_j]$. Note that if there 
 exists $i \ne j$ such that $W \cap [u_i] = W \cap N[u_{k+1}]$, then $W$ does not separate $u_i$ and $u_j$, which is a contradiction. Thus, $u_j$ 
 is the only vertex of $R(a_0) - \{u_{k+1}\}$ that satisfies $W \cap N[u_{k+1}] = W \cap N[u_j]$. There exists a vertex in $V(G) - (W \cup C')$ that 
 is adjacent to exactly one of $u_{k+1}$ or $u_j$ for otherwise $u_{k+1}$ and $u_j$ are twins in $G$. Assume first that there exists 
 $z \in (N[u_j] - N[u_{k+1}]) - (W \cup C')$. It follows that $W' = W \cup \{z\}$ separates every pair of vertices of $R(a_0)$. Otherwise, there 
 exists $z \in (N[u_{k+1}] - N[u_j]) - (W \cup C')$ and $W' = W \cup \{z\}$ separates every pair of vertices in $R(a_0)$. In either case, 
 we have found a set of $|R(a_0)|$ vertices in $V(G) - C'$ that dominates and separates each pair of vertices in $R(a_0)$.\smallqed

 {\underline{Claim 2}} We can choose a set of $|R(a_i)| - 1$ vertices from $V(G)- C'$ that separates each pair of vertices of $R(a_i)$ for $i \in [m]$.
 \vskip2mm
 {\underline{Proof}} First, let $m_1 < i \le m$. Note that $R(a_i) = \{a_i\}$ in which case there is no need to choose any vertices to 
 separate $a_i$ from itself. Now suppose $i \in [m_1]$. As in the proof of Claim 2, we proceed by induction on the cardinality of $R(a_i)$. 
 Suppose first that $R(a_i) = \{a_i, v\}$. If $a_i$ is not adjacent to $v$, then it follows that $a_i$ and $v$ are not vertices of $C'$. 
 Moreover, $a_i$ separates $a_i$ and $v$. On the other hand, If $a_i$ is adjacent to $v$, then either there exists $w \in (V(G) - C') \cap (N[a_i]- N[v])$ 
 or there exists $w \in (V(G) - C') \cap (N[v]- N[a_i])$ for otherwise $a_i$ and $v$ are twins in $G$. In either case, $w$ separates $a_i$ and $v$. 
 So we shall assume that when $|R(a_i)| = k$, there exists a set of $k-1$ vertices in $V(G) - C'$ that separates each pair of vertices of $R(a_i)$.

 Suppose that $R(a_i) = \{u_1, \dots u_{k+1}\}$. By the inductive hypothesis, there exists a set $W \subseteq  V(G) - C'$ of cardinality $k-1$ that 
 separates each pair of vertices in $R(a_i) - \{u_{k+1}\}$. If $W$ separates $u_{k+1}$ and $u_j$ for all $j \in [k]$, then we are done. So assume 
 that for some $j \in [k]$ that $W \cap N[u_j] = W \cap N[u_{k+1}]$. If there exists $1 \le i\le k,  i\ne j$ such that $W \cap N[u_i] = W \cap N[u_{k+1}]$, 
 then $W$ does not separate $u_i$ and $u_j$, which is a contradiction. Therefore, $u_j$ is the only vertex of $R(a_i) - \{u_{k+1}\}$ that satisfies 
 $W \cap N[u_j] = W \cap N[u_{k+1}]$. Since $u_j$ and $u_{k+1}$ are not twins in $G$, then there exists $z \in V(G) - (W\cup C')$ that is adjacent 
 to exactly one of $u_j$ or $u_{k+1}$. Thus, $W \cup \{z\}$ separates every pair of vertices of $R(a_i)$ and $|W \cup \{z\}| = k$.
 \smallqed

Finally, choose a minimal set $W$ of vertices from $V(G) - C'$ that separates every pair of vertices from $R(a_i)$ for all $0 \le i \le m$ 
and dominates $R(a_0)$, which we know exists from Claim 1 and Claim 2. Note that $W \cup C'$ dominates every vertex of $V(G)$ since every 
vertex $v$ not dominated by $C'$ satisfies $v \in R(a_0)$ and $W$ dominates $R(a_0)$. Next, note that if $C'$ does not separate a pair of 
vertices, say $x, y \in V(G)$, then there exists $a_i$ such that $\{x, y\} \subseteq R(a_i)$ for some $0 \le i \le m$. In this case, some 
vertex of $W$ separates $x$ and $y$. Thus, $W \cup C'$ is an ID code of $G$ and $\gid(G) \le |W \cup C'|$. We claim that 
$|W \cup C'| \le |C \cap (V(G) \times N_H[h])|$. Indeed, if $(u,h)\in V(G^h)$ where $u \in R(a_0)$, then there exists $h' \in V(H)$ such 
that $hh'\in E(H)$ and $(u,h') \in C$ since $C$ is an ID code of $G \Box H$. Moreover, for each  $1 \le i \le m$, consider the set 
$S_i = \{(u,h)\in V(G^h) : u \in R(a_i)\}$. Since $C$ is an ID code of $G \Box H$, $|C \cap (V(G) \times N_H(h))|\ge |S_i|-1$. Thus, 
$|W| \le |C \cap (V(G) \times N_H(h))|$, which implies that
\begin{eqnarray*}
|W \cup C'| &=& |W| + |C'|\\
&\le& |C \cap (V(G) \times N_H(h))| + |C \cap V(G^h)|\\
&=& |C \cap (V(G) \times N_H[h])|.
\end{eqnarray*}

Notice that the above argument shows that there exist at least $\gid(G)$ codewords of $C$ in $V(G) \times N_H[h]$. Therefore, if we choose 
a maximum $2$-packing, $T$, of $H$ and apply the same argument to each vertex of $T$, then the desired result follows.
\end{proof}

We call the reader's attention to the fact that Theorem~\ref{thm:lowerbound} does not require that $H$ be twin-free. Thus, an immediate 
consequence of Theorem~\ref{thm:lowerbound} is the following.

\begin{cor}
For any twin-free graphs $G$ and $H$, \[\gid(G\Box H) \ge \max\{\gid(G)\rho_2(H), \rho_2(G)\gid(H)\}.\]
\end{cor}

Next, we show that the bound given in Theorem~\ref{thm:lowerbound} is indeed sharp. For the remainder of this section, we consider only Cartesian products of the form $G \Box K_2$. Note that by Corollary~\ref{cor:upper}, $\gid(G \Box K_2) > \gid(G)$ when $\gid(G) \le 3$ as $\gid(G) \le |V(G)|-1$ for all graphs. So the first case we consider is when $\gid(G) = 4$. 

Surprisingly, the class of graphs for which $\gid(G\Box K_2) = \gid(G) = 4$ is a subclass of the graphs which satisfy $\gid(G) = |V(G)| - 1$. Foucaud et al. classified all such graphs that satisfy $\gid(G) = |V(G)| - 1$ in \cite{extremalID}. For ease of reference, we include the description of this class of graphs here along with their result.

For any integer $k\ge 1$, let $A_k = (V_k, E_k)$ be the graph with vertex set $V_k = \{x_1, \dots, x_{2k}\}$ and edge set $E_k = \{x_ix_j : |i-j| \le k-1\}$. So for $k\ge 2$, $A_k = P_{2k}^{k-1}$ and $A_1 = \overline{K_2}$. Let $\mathcal{A}$ be the closure of
$\{A_i : i \in \mathbb{N} \}$ with respect to the join operation $\bowtie$. Figure~\ref{fig:extremal} depicts several graphs in $\mathcal{A}$.
\vskip2mm

\begin{figure}[h!]
\begin{center}
\begin{tikzpicture}

\vertex(a1) at (0,0)[scale=.65, fill=black]{};
\vertex(a2) at (1,0)[scale=.65, fill=black]{};
\node(1) at (0.5,-1)[]{$A_1$};
\draw(1);
\vertex(b1) at (4,0)[scale=.65, fill=black]{};
\vertex(b2) at (5,0)[scale=.65, fill=black]{};
\vertex(b3) at (6,0)[scale=.65, fill=black]{};
\vertex(b4) at (7,0)[scale=.65, fill=black]{};
\node(2) at (5.5,-1)[]{$A_2$};
\draw(2);
\vertex(c1) at (6,-3.5)[scale=.65, fill=black]{};
\vertex(c2) at (7,-3.5)[scale=.65, fill=black]{};
\vertex(c3) at (8,-3.5)[scale=.65, fill=black]{};
\vertex(c4) at (9,-3.5)[scale=.65, fill=black]{};
\vertex(c5) at (10,-3.5)[scale=.65, fill=black]{};
\vertex(c6) at (11,-3.5)[scale=.65, fill=black]{};
\node(3) at (8.5,-4)[]{$A_3$};
\draw(3);
\vertex(d1) at (1, -3)[scale=.65, fill=black]{};
\vertex(d2) at (2,-3)[scale=.65, fill=black]{};
\vertex(d3) at (0,-4)[scale=.65, fill=black]{};
\vertex(d4) at (1,-4)[scale=.65, fill=black]{};
\vertex(d5) at (2,-4)[scale=.65, fill=black]{};
\vertex(d6) at (3,-4)[scale=.65, fill=black]{};
\node(4) at (-1, -3.5)[]{$A_1 \bowtie A_2$};
\draw(4);
\vertex(e1) at (10, -.5)[scale=.65, fill=black]{};
\vertex(e2) at (10.5, 0)[scale=.65, fill=black]{};
\vertex(e3) at (11.5, 0)[scale=.65, fill=black]{};
\vertex(e4) at (12, -.5)[scale=.65, fill=black]{};
\vertex(e5) at (10.5, -1.25)[scale=.65, fill=black]{};
\vertex(e6) at (11.5, -1.25)[scale=.65, fill=black]{};
\node(5) at (11, -1.75)[]{$A_1\bowtie A_1 \bowtie A_1$};
\draw(5);

\path

	(b1) edge (b2)
	(b2) edge (b3)
	(b3) edge (b4)
	(c1) edge (c2)
	(c1) edge[bend left=20] (c3)
	(c2) edge (c3)
	(c2) edge[bend left=20] (c4)
	(c3) edge (c4)
	(c3) edge[bend left=20] (c5)
	(c4) edge (c5)
	(c4) edge[bend left=20] (c6)
	(c5) edge (c6)
	(d1) edge (d3)
	(d1) edge (d4)
	(d1) edge (d5)
	(d1) edge (d6)
	(d2) edge (d3)
	(d2) edge (d4)
	(d2) edge (d5)
	(d2) edge (d6)
	(d3) edge (d4)
	(d4) edge (d5)
	(d5) edge (d6)
	(e1) edge (e3)
	(e1) edge (e4)
	(e1) edge (e5)
	(e1) edge (e6)
	(e2) edge (e3)
	(e2) edge (e4)
	(e2) edge (e5)
	(e2) edge (e6)
	(e3) edge (e5)
	(e3) edge (e6)
	(e4) edge (e5)
	(e4) edge (e6);

\end{tikzpicture}
\caption{Examples of graphs in $\mathcal{A}$}
\label{fig:extremal}
\end{center}
\end{figure}
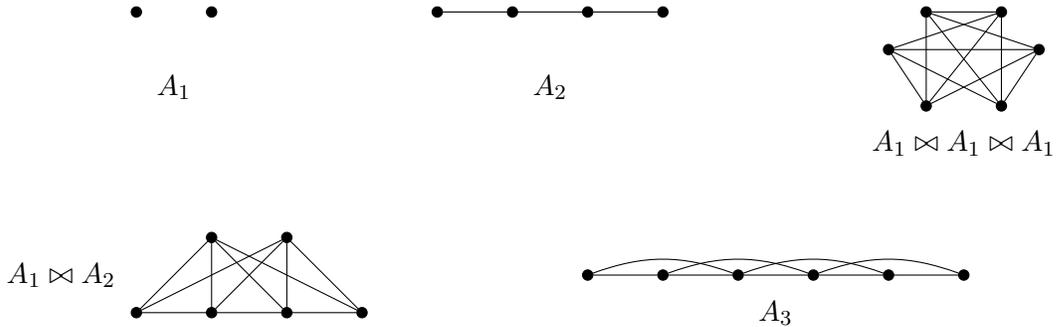

\begin{theorem}[\cite{extremalID}]\label{thm:extremal} Given a connected graph $G$, we have $\gid(G) = |V(G)|-1$ if and only if $G \in \{K_{1,t} \ \mid t \ge 2\} \cup \mathcal{A} \cup (\mathcal{A} \bowtie K_1)$ and $G \not\cong A_1$.
\end{theorem}

We now show that a subclass of $\mathcal{A}$ contains precisely those graphs for which $\gid(G\Box K_2) = \gid(G) = 4$.

\begin{theorem}\label{thm:gid4} For any connected twin-free graph $G$ such that $\gid(G)=4$, $\gid(G \Box K_2) =\gid(G)$ if and only if $G \in \mathcal{A}\bowtie K_1$.
\end{theorem}

\begin{proof} Notice that if $G \in \mathcal{A}\bowtie K_1$ with $\gid(G) = 4$, then $G = A_1 \bowtie A_1 \bowtie K_1$ or $G = A_2 \bowtie K_1$. In either case, we represent the vertices of $A_1 \bowtie A_1$ or $A_2$ by $x_1, x_2, x_3, x_4$. If $G = A_1 \bowtie A_1 \bowtie K_1$, then $C = \{x_1^1, x_2^1, x_3^2, x_4^2\}$ is an ID code of $G \Box K_2$. If $G = A_2 \bowtie K_2$, then $C = \{x_1^1, x_3^1, x_2^2, x_4^2\}$ is an ID code of $G \Box K_2$. Thus, $\gid(G\Box K_2) \le 4$. An application of Theorem~\ref{thm:lowerbound} yields $\gid(G\Box K_2) \ge \gid(G) = 4$. Therefore, $\gid(G\Box K_2) = 4$.

We now show the other direction. That is, let $G$ be a connected twin-free graph such that $\gid(G) = 4 = \gid(G\Box K_2)$. Let $C$ be a minimum ID code of $G\Box K_2$ and partition the projection of $C$ onto $V(G)$, $p_G[C]$, as
\begin{eqnarray*}
C_1 &=& \{v \in V(G) : v^1 \in C \text{ and} \ v^2 \not\in C\}\\
C_2 &=& \{v \in V(G): v^1 \not\in C \text{ and}\  v^2 \in C\}\\
D &=& \{v\in V(G) : v^1 \in C \text{ and} \ v^2 \in C\}.
\end{eqnarray*}

Suppose first that $|C_1| = 1$, $|D|=0$, and let $C_1 = \{v\}$. Thus, $I_C(v^1) = \{v^1\}$, which implies for every $u \in V(G)-\{v\}$, $u^2 \in C$. It follows that $|V(G)| = 4$, which contradicts the assumption that $\gid(G) = 4$. On the other hand, suppose $|C_1|=0$, $|D|=1$ and $D= \{v\}$. There exist precisely two vertices, say $x$ and $y$ in $G$ such that $x^2 \in I_C(x^1)$ and $y^2 \in I_C(y^1)$ as $|C_2|=2$. Every $w^1 \in V(G^1)-\{v^1, x^1, y^1\}$ is dominated only by $v^1$ and this implies that $|V(G)|=4$, which is another contradiction. Thus, $|C_1 \cup D| = 2$ and similarly $|C_2 \cup D| = 2$.

\begin{enumerate}
\item[(1)] Suppose that $|D| = 2$ and let $D = \{u,v\}$. It follows that the order of $G$ is at most $5$, and since $\gid(G) = 4$, we have $|V(G)| = 5$. Theorem~\ref{thm:extremal} guarantees that $G \in \{K_{1,4}, A_1 \bowtie A_1 \bowtie K_1, A_2 \bowtie K_1\}$, pictured below in Figure~\ref{fig:order5}. Note that $uv \in E(G)$ since the subgraph induced by $C$ contains no isolated edge. Furthermore, since $|V(G)|=5$, there exists $w \in V(G)$ such that $w$ is adjacent to both $u$ and $v$. Therefore, $G$ contains a triangle and it follows that $G \in \{ A_1 \bowtie A_1 \bowtie K_1, A_2 \bowtie K_1\}$.

\item[(2)] Suppose that $|D|=1$, meaning $|C_1|=1=|C_2|$, and let $C_1 = \{u\}$, $D=\{v\}$, and $C_2= \{w\}$. Since the subgraph induced by $C$ contains no isolated edges, we may assume without loss of generality that $uv \in E(G)$. This immediately implies that $|V(G)| = 5$ and there exist vertices $x$ and $y$ in $G$ such that $I_C(x^1) = \{u^1\}$ and $I_C(y^1) = \{v^1\}$. Therefore, $G \in \{ A_1 \bowtie A_1 \bowtie K_1, A_2 \bowtie K_1\}$ since the subgraph induced by $x, y, u$, and $v$ is a path.

\item[(3)] Suppose that $|D| = 0$, $|C_1|=2=|C_2|$, and let $C_1 = \{u,v\}$ and $C_2 = \{x,y\}$. Note that $uv \not\in E(G)$ and $xy \not\in E(G)$ since the subgraph induced by $C$ contains no isolated edge. Thus, for any $w \in V(G) - (C_1 \cup C_2)$, $I_C(w^1) = \{u^1, v^1\}$ and $I_C(w^2) = \{x^2, y^2\}$. So $|V(G)|=5$, $N[w] = V(G)$, and $G \in \{K_{1,4}, A_1 \bowtie A_1 \bowtie K_1, A_2 \bowtie K_1\}$. Also, $I_C(x^1) \ne \{x^2\}$ so $x$ has a neighbor in $C_1$. Therefore, we may conclude that $G \in \{ A_1 \bowtie A_1 \bowtie K_1, A_2 \bowtie K_1\}$.

\end{enumerate}
\end{proof}

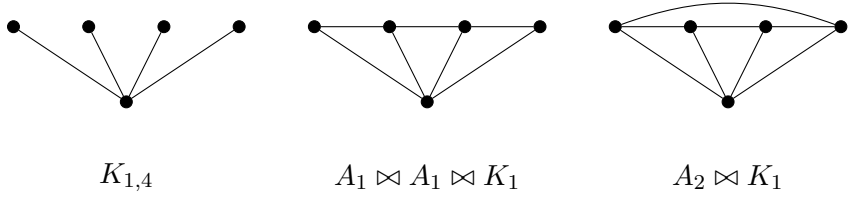
\begin{figure}[h!]
\begin{center}
\begin{tikzpicture}[auto]

	\vertex (1a) at (0,0) [scale=.75pt,fill=black]{};
	\vertex (2a) at (-1.5,1) [scale=.75pt,fill=black]{};
	\vertex (3a) at (-.5,1) [scale=.75pt,fill=black]{};
	\vertex (4a) at (.5,1) [scale=.75pt,fill=black]{};
	\vertex (5a) at (1.5,1) [scale=.75pt,fill=black]{};
	\vertex (1b) at (4,0) [scale=.75pt,fill=black]{};
	\vertex (2b) at (2.5,1) [scale=.75pt,fill=black]{};
	\vertex (3b) at (3.5,1) [scale=.75pt,fill=black]{};
	\vertex (4b) at (4.5,1) [scale=.75pt,fill=black]{};
	\vertex (5b) at (5.5,1) [scale=.75pt,fill=black]{};	
	\vertex (1c) at (8,0) [scale=.75pt,fill=black]{};
	\vertex (2c) at (6.5,1) [scale=.75pt,fill=black]{};
	\vertex (3c) at (7.5,1) [scale=.75pt,fill=black]{};
	\vertex (4c) at (8.5,1) [scale=.75pt,fill=black]{};
	\vertex (5c) at (9.5,1) [scale=.75pt,fill=black]{};	
	\node(a) at (0,-1)[]{$K_{1,4}$};
	\draw(a);
	\node(b) at (4,-1)[]{$A_1 \bowtie A_1 \bowtie K_1$};
	\draw(b);
	\node(c) at (8,-1)[]{$A_2 \bowtie K_1$};
	\draw(c);
	\path
		(1a) edge (2a)
		(1a) edge (3a)
		(1a) edge (4a)
		(1a) edge (5a)
		(1b) edge (2b)
		(1b) edge (3b)
                  (1b) edge (4b)
                  (1b) edge (5b)
                  (2b) edge (3b)
                  (3b) edge (4b)
                  (4b) edge (5b)
                  (1c) edge (2c)
                  (1c) edge (3c)
                  (1c) edge (4c)
                  (1c) edge (5c)
                  (2c) edge (3c)
                  (3c) edge (4c)
                  (4c) edge (5c)
                  (2c) edge[bend left=20] (5c)
		;

\end{tikzpicture}
\end{center}
\caption{Graphs of order $5$ with ID code number $4$}
\label{fig:order5}
\end{figure}

Based on the above result, we next show that for any integer $k \ge 4$, there exists a graph $G$ such that $\gid(G \Box K_2) = \gid(G) = k$.

\begin{theorem}\label{thm:lowerclass} If $G \in \mathcal{A} \cup (\mathcal{A} \bowtie K_1)$ has order at least $5$, then $\gid(G \Box K_2) = \gid(G)$. Moreover, if $G = G_1 \bowtie G_2$ where $G_1, G_2  \in \mathcal{A} \cup (\mathcal{A} \bowtie K_1) - \{A_1, A_2\}$, then \\
$\gid((G_1 \bowtie G_2) \Box K_2) = \gid(G_1 \Box K_2) + \gid(G_2 \Box K_2) + 1$.
\end{theorem}

\begin{proof}
For the time being, assume that $G \in \mathcal{A}$. We proceed by induction. Write $G = G_1 \bowtie \cdots \bowtie G_m$ where each $G_i \in \{A_j : j \in \mathbb{N}\}$. Suppose first that $G = A_k$ where $k>2$. That is, $G = P_{2k}^{k-1}$ for some $k \ge 3$. We show that \[C = \{x_1^1, \dots, x_{k-1}^1, x_{k+1}^1, \dots, x_{2k-2}^1, x_k^2, x_{2k}^2\}\] is an ID code for $G\Box K_2$ of order $2k-1$. Figure~\ref{fig:construct} (a) depicts $C$ for $A_5 \Box K_2$. Let $u$ and $v$ be any pair of vertices in $G\Box K_2$. One can easily verify that $C$ is a dominating set for $G\Box K_2$ and if $u \in V(G^1)$ and $v \in V(G^2)$, then $C$ separates $u$ and $v$. We check all remaining cases.
\begin{figure}[ht!]
\begin{subfigure}[]{.5\textwidth}
\centering
\begin{tikzpicture}[scale=1.0,style=thick]
	\vertex (0) at (0,0) [scale=.75pt,fill=black]{};
	\vertex (1) at (0.5,0) [scale=.75pt,fill=black]{};
	\vertex (2) at (1,0) [scale=.75pt,fill=black]{};
	\vertex (3) at (1.5,0) [scale=.75pt,fill=black]{};
	\vertex (4) at (2,0) [ scale=.75pt]{};
	\vertex (5) at (2.5,0) [scale=.75pt,fill=black]{};
	\vertex (6) at (3,0) [scale=.75pt,fill=black]{};
	\vertex (7) at (3.5,0) [scale=.75pt,fill=black]{};
	\vertex (8) at (4,0) [scale=.75pt]{};
	\vertex (9) at (4.5,0) [scale=.75pt]{};

	\vertex (0a) at (0,1) [scale=.75pt]{};
	\vertex (1a) at (0.5,1) [scale=.75pt]{};
	\vertex (2a) at (1,1) [scale=.75pt]{};
	\vertex (3a) at (1.5,1) [scale=.75pt]{};
	\vertex (4a) at (2,1) [scale=.75pt,fill=black]{};
	\vertex (5a) at (2.5,1) [scale=.75pt]{};
	\vertex (6a) at (3,1) [scale=.75pt]{};
	\vertex (7a) at (3.5,1)[scale=.75pt]{};
	\vertex (8a) at (4,1)[scale=.75pt]{};
	\vertex (9a) at (4.5,1)[scale=.75pt,fill=black]{};

\end{tikzpicture}
\caption{ID code for $A_5 \Box K_2$}
\end{subfigure}~
\begin{subfigure}[]{.5\textwidth}
\centering
\begin{tikzpicture}[scale=1.0,style=thick]
	\vertex (0) at (0,0) [scale=.75pt,fill=black]{};
	\vertex (1) at (0.5,0) [scale=.75pt]{};
	\vertex (2) at (1.5,0) [scale=.75pt,fill=black]{};
	\vertex (3) at (2,0) [scale=.75pt]{};
	\vertex (4) at (2.5,0) [ scale=.75pt,fill=black]{};
	\vertex (5) at (3,0) [scale=.75pt]{};
	\vertex (6) at (3.5,0) [scale=.75pt,fill=black]{};
	\vertex (7) at (4,0) [scale=.75pt]{};
	\vertex (8) at (4.5,0) [scale=.75pt,fill=black]{};
	\vertex (9) at (5,0) [scale=.75pt]{};

	\vertex (0a) at (0,1) [scale=.75pt]{};
	\vertex (1a) at (.5,1) [scale=.75pt]{};
	\vertex (2a) at (1.5,1) [scale=.75pt]{};
	\vertex (3a) at (2,1) [scale=.75pt,fill=black]{};
	\vertex (4a) at (2.5,1) [scale=.75pt]{};
	\vertex (5a) at (3,1) [scale=.75pt,fill=black]{};
	\vertex (6a) at (3.5,1) [scale=.75pt]{};
	\vertex (7a) at (4,1) [scale=.75pt,fill=black]{};
	\vertex (8a) at (4.5,1)[scale=.75pt]{};
	\vertex (9a) at (5,1)[scale=.75pt,fill=black]{};

\end{tikzpicture}
\caption{ID code for $(A_1 \bowtie A_4) \Box K_2$}
\end{subfigure}
\caption{Examples of ID codes of $A_k \Box K_2$ or $(A_k \bowtie A_{\ell})\Box K_2$}
\label{fig:construct}
\end{figure}
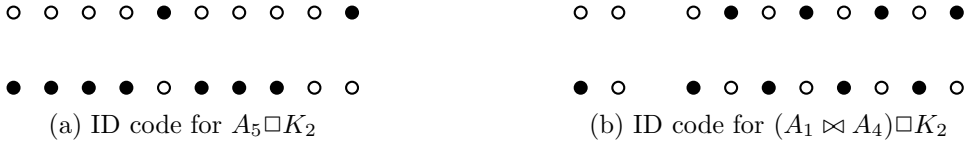

\begin{figure}[h!]
\begin{center}
\begin{tikzpicture}[scale=1.0,style=thick]
	\vertex (0) at (0,0) [scale=.75pt,fill=black]{};
	\vertex (1) at (0.5,0) [scale=.75pt,fill=black]{};
	\vertex (2) at (1,0) [scale=.75pt,fill=black]{};
	\vertex (3) at (1.5,0) [scale=.75pt]{};
	\vertex (4) at (2.5,0) [ scale=.75pt,fill=black]{};
	\vertex (5) at (3,0) [scale=.75pt]{};
	\vertex (6) at (3.5,0) [scale=.75pt,fill=black]{};
	\vertex (7) at (4,0) [scale=.75pt]{};
	\vertex (8) at (4.5,0) [scale=.75pt,fill=black]{};
	\vertex (9) at (5,0) [scale=.75pt]{};
	\vertex (10) at (5.5,0) [scale=.75pt,fill=black]{};
	\vertex (11) at (6,0) [scale=.75pt]{};
	
	\vertex (0a) at (0,1) [scale=.75pt]{};
	\vertex (1a) at (.5,1) [scale=.75pt]{};
	\vertex (2a) at (1,1) [scale=.75pt]{};
	\vertex (3a) at (1.5,1) [scale=.75pt]{};
	\vertex (4a) at (2.5,1) [scale=.75pt]{};
	\vertex (5a) at (3,1) [scale=.75pt,fill=black]{};
	\vertex (6a) at (3.5,1) [scale=.75pt]{};
	\vertex (7a) at (4,1) [scale=.75pt,fill=black]{};
	\vertex (8a) at (4.5,1)[scale=.75pt]{};
	\vertex (9a) at (5,1)[scale=.75pt,fill=black]{};
	\vertex (10a) at (5.5,1) [scale=.75pt]{};
	\vertex (11a) at (6,1) [scale=.75pt,fill=black]{};
\end{tikzpicture}
\caption{Example of ID code of $(A_2 \bowtie A_4)\Box K_2$}
\label{fig:construct2}
\end{center}
\end{figure}
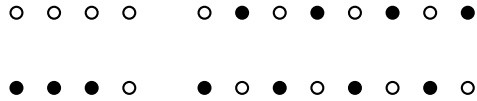

Suppose first that $u = x_i^1$ and $v = x_j^1$ where $1 \le i < j \le 2k$. If $1\le i < j \le k-1$, then $x_{j+(k-1)}^1$ separates $u$ and $v$. 
If $1 \le i \le k-1$ and $j=k$, then $x_k^2$ separates $u$ and $v$. If $1\le i \le k$ and $k+1 \le j \le 2k$, then $x_1^1$ separates $u$ and $v$. 
If $k \le i < j \le 2k-1$, then $x_{i-(k-1)}^1$ separates $u$ and $v$. If $k \le i \le 2k-1$ and $j = 2k$, then $x_{2k}^2$ separates $u$ and $v$.

Next, suppose that $u = x_i^2$ and $v = x_j^2$ where $1 \le i<j \le 2k$. If $i \notin \{k, 2k-1\}$, then $x_i^1$ separates $u$ and $v$. If $i = k$ 
and $j = 2k-1$, then $x_{2k}^2$ separates $u$ and $v$. If $i = k$ and $j = 2k$, then $u$ separates $u$ and $v$. Finally, if $i = 2k-1$ and $j = 2k$, 
then $x_k^2$ separates $u$ and $v$. Thus, $C$ is an ID code of $G$, and we have shown that $\gid(G\Box K_2) \le 2k-1$. On the other hand, 
$G \in \mathcal{A}$ so $\gid(G) = 2k-1$ by Theorem~\ref{thm:extremal}. Thus, by Theorem~\ref{thm:lowerbound} $\gid(G \Box K_2) \ge \gid(G) = 2k-1$, 
which implies that $\gid(G\Box K_2) = \gid(G)$.

Next, suppose $G = A_k \bowtie A_{\ell}$ where $k \in [ \ell]$. Since $G$ has order at least $5$, $\ell \ge 2$. Let $x_1, \dots, x_{2k}$ represent 
the vertices of $A_k$ and $y_1, \dots, y_{2\ell}$ represent the vertices of $A_{\ell}$. We construct an ID code of $G\Box K_2$ based on the 
following three cases, where in each case $A = \{y_{i}^1 : i = 2j+1 \text{ for $0 \le j \le \ell-1$}\}$ and $B = \{y_i^2 : i = 2j \text{ for $1 \le j \le \ell$}\}$.
\begin{enumerate}
\item[1.] Suppose $k = 1$. Note that $\ell \ge 2$ since the order of $G$ is at least $6$. We show that $C = A \cup B \cup \{x_1^1\}$ is an ID 
code for $G\Box K_2$. Figure~\ref{fig:construct} (b) depicts $C$ for $(A_1 \bowtie A_4) \Box K_2$. For $i \in [2]$ and each 
$v^i \in V(G^i)$, $|N_{G^i}[v^i] \cap C | \ge 2$, and it follows that $C$ separates any vertex in $G^1$ from any vertex in $G^2$. 
Note that $x_1^1$ separates $x_1^1$ and $x_2^1$, and $x_1^1$ separates $x_1^2$ from any other vertex of $G^2$. Next, for $j \in [\ell]$, $y_{2j-1}^1$ 
separates $y_{2j-1}^2$ from every other vertex of $G^2$. By definition of $A_{\ell}$, $N[y_{2i}^2] \cap B \ne N[y_{2j}^2] \cap B$ for $1 \le i <j \le \ell$. 
Since $N[x_2^2] \cap B = B$, $C$ separates $x_2^2$ from $y_{2i}^2$ for $i \in [ \ell]$. Similarly, $C$ separates any two vertices in $G^1$. 
Therefore, $C$ is an ID code of $G\Box K_2$, which implies $\gid(G\Box K_2) \le \gid(G)$.

\item[2.] Suppose $k \ge 2$. Let $T = \{ x_i^1: 1 \le i \le 2k-1\}$. We show that $C = A \cup B \cup T$ is an ID code of $G\Box K_2$. Figure~\ref{fig:construct2} 
depicts $C$ for $(A_2 \bowtie A_4)\Box K_2$. As in Case 1, $C$ separates any vertex in $G^1$ from any vertex in $G^2$. Note that for $i \in [ 2k-1]$, $x_i^1$ 
separates $x_i^2$ from any other vertex of $G^2$. Next, for $j \in [\ell]$, $y_{2j-1}^1$ separates $y_{2j-1}^2$ from every other vertex of $G^2$. By definition 
of $A_{\ell}$, $N[y_{2i}^2] \cap B \ne N[y_{2j}^2] \cap B$ for $1 \le i <j \le \ell$. Since $N[x_{2k}^2] \cap B = B$, $C$ separates $x_{2k}^2$ from $y_{2i}^2$ 
for $i \in [\ell]$. Similarly, $C$ separates any two vertices. Moreover, $C$ is an ID code of the subgraph induced by $\{x_i^1 : 1 \le i \le 2k\}$. For 
each $j  \in[\ell]$, $y_{2j}^2$ separates $y_{2j}^1$ from every other vertex in $G^1$. By definition $A_{\ell}$, 
$N[y_{2i-1}^1] \cap A \ne N[y_{2j-1}^1] \cap A$ for $1 \le i < j \le \ell$. Furthermore, this shows that $C$ separates $y_{2i-1}^1$ from $x_j^1$ 
where $i \in [ \ell]$ and $j \in [ 2k]$ since $N[x_j^1] \cap A = A$. Therefore, $C$ is an ID code of $G \Box K_2$, which implies 
$\gid(G \Box K_2) \le \gid(G)$.

\end{enumerate}
Finally, note that by Theorem~\ref{thm:lowerbound}, we know $\gid(G\Box K_2) \ge \gid(G)$, which implies $\gid(G \Box K_2) = \gid(G)$. This concludes the base cases.

Suppose now that $r \ge2$ and that if $G  = G_1 \bowtie \cdots \bowtie G_r$ where each $G_j \in \{A_i :i \in \mathbb{N}\}$, then $\gid(G\Box K_2) = \gid(G)$. 
Now consider $H = A_s \bowtie G_1 \bowtie \cdots \bowtie G_r$ where $s \ge 1$.  Let $G  = G_1 \bowtie \cdots \bowtie G_r$.  We can assume with no loss of 
generality that in the expansion $G = G_1 \bowtie \cdots \bowtie G_r$ that $|V(G_a)| \le |V(G_b)|$ when $a<b$. Thus, if $|V(A_s)|>|V(G_1)|$, we can let 
$H=G_1 \bowtie (A_s \bowtie G_2 \bowtie \cdots \bowtie G_s)$.

Suppose first that $H = A_1 \bowtie G$. Let $C$ be a minimum ID code for $G\Box K_2$ and let $x_1, x_2$ represent the vertices of $A_1$. We claim that 
$C' = C \cup \{x_1^1, x_2^1\}$ is an ID code for $H\Box K_2$. Clearly $C'$ dominates $H\Box K_2$, and any pair of vertices in $V(G\Box K_2)$ are separated 
by $C$, and therefore by $C'$. Suppose that $u, v \in \{x_1^1, x_2^1, x_1^2, x_2^2\}$. Note that if $u \in \{x_1^1, x_2^1\}$, then $I_{C'}(u) \cap G^1 \ne \emptyset$.
 Similarly, if $u \in \{x_1^2, x_2^2\}$, then $I_{C'}(u) \cap G^2 \ne \emptyset$. Thus, if $u \in \{x_1^1, x_2^1\}$ and $v \in \{x_1^2, x_2^2\}$, then 
 $I_{C'}(u) \ne I_{C'}(v)$. If $u = x_1^1$ and $v = x_2^1$, then $u \in I_{C'}(u)$ but $u \not\in I_{C'}(v)$. Similarly, if $u = x_1^2$ and $v= x_2^2$, 
 then $x_1^1 \in I_{C'}(u)$ but $x_1^1 \not\in I_{C'}(v)$. Finally, if $u \in \{x_1^1, x_2^1, x_1^2, x_2^2\}$ and $v \notin \{x_1^1, x_2^1, x_1^2, x_2^2\}$, 
 then one of $x_1^1$ or $x_2^1$ separates $u$ and $v$. Thus, $C'$ is an ID code of $H\Box K_2$ and by the inductive assumption and Theorem~\ref{thm:extremal}
$$|C'|=2 + |C|= 2+ \gid(G \Box K_2) = 2 + \gid(G)=|V(H)|-1=\gid(H)\,.$$
This implies that  $\gid(H \Box K_2) \le \gid(H)$.  An application of Theorem~\ref{thm:lowerbound} gives $\gid(H \Box K_2) = \gid(H)$.

Next, suppose that $H = A_i \bowtie G$ where $i \ge 2$. Let $C$ be a minimum ID code of $G \Box K_2$.
We claim that $C' = C \cup A_i^1$ is an ID code of $H \Box K_2$. Clearly $C'$ dominates $H \Box K_2$, and any pair of vertices in $V(G \Box K_2)$ 
are separated by $C$, and therefore by $C'$. Next, note that $A_i^1$ is an ID code of $A_i \Box K_2$. Thus, $C'$ separates every pair of vertices 
in $A_i^1 \cup A_i^2$. Finally, suppose that $u \in A_i^1 \cup A_i^2$ and $v \in G^1 \cup G^2$. If $u \in A_i^1$ and $v \in G^2$, then $u$ separates 
$u$ and $v$. Similarly, if $u \in A_i^2$ and $v \in G^2$, then some vertex of $A_i^1$ separates $u$ and $v$. So assume that $v \in G^1$. No vertex of 
$A_i^1 \cup A_i^2$ is adjacent to every vertex of $A_i^1$, but $A_i^1 \subset I_{C'}(v)$.  Hence $C'$ separates every pair of vertices in $H \Box K_2$, 
and consequently $C'$ is an ID code of $H \Box K_2$.  In a manner similar to that in the previous case, by using our induction assumption together 
with Theorems~\ref{thm:lowerbound} and  \ref{thm:extremal} we get that $\gid(H \Box K_2) = \gid(H)$.

Next, suppose that $G \in \mathcal{A} \bowtie K_1$. As above, we proceed by induction with base case $G = A_k \bowtie K_1$ where $k \in \mathbb{N}$. 
Note that $k \ge 2$ since the order of $G$ is at least $5$. If $k=2$, then we are done by Theorem~\ref{thm:gid4}. If $k>2$, then one can easily verify 
that $C = A \cup B$ where $A = \{x_{2j-1}^1 :  j \in  [k]\}$ and $B = \{x_{2j}^2 : j \in [ k]\}$ is an ID code for $G \Box K_2$. Thus, 
$\gid(G \Box K_2) \le |V(G)| - 1$ and by Theorem~\ref{thm:lowerbound}, we have $\gid(G\Box K_2) = \gid(G)$. We now assume that for some 
$m \ge 2$,  $\gid(G\Box K_2) = \gid(G)$ if
$G = G_1 \bowtie G_2 \bowtie \cdots \bowtie G_{m-1} \bowtie K_1$ where each $G_j \in \{A_i : i \in \mathbb{N}\}$.

 Suppose $H = G_1 \bowtie \cdots \bowtie G_m \bowtie K_1$ where $G_j \in \{A_i : i \in \mathbb{N}\}$.  Label the vertices 
 of $G_j = A_{t_j}$, $t_j \ge 1$, as $x_{j,1}, \dots, x_{j,2t_j}$ and let $y$ be the vertex of $K_1$. For each $j$ where $G_j = A_1$, let 
 $C_j = \{x_{j,1}^1, x_{j, 2}^1\}$ if $j$ is odd and let $C_j = \{x_{j,1}^2, x_{j,2}^2\}$ if $j$ is even. For each $G_j = A_{t_j}$ 
 where $t_j>1$, let  \[C_{j,1} = \{x_{j,2k-1}^1 :  k \in  [t_j]\},\]
 and
 \[C_{j,2} = \{x_{j,2k}^2 :  k \in  [t_j]\}.\]
Finally, let $C_j = C_{j,1} \cup C_{j,2}$. We show that $C = \cup_{j=1}^m C_j$ is an ID code for $G \Box K_2$.  Let $u$ and $v$ be
any two vertices in $V(G \Box K_2)$.  If $u\in G^1$ and $v \in G^2$, then $|I_C(u) \cap V(G^1)| \ge 2$ and $|I_C(v) \cap V(G^2)| \ge 2$,
which implies that $C$ separates $u$ and $v$. Now suppose that $u$ and $v$ belong to $G^1$.  If $u=y^1$, then $I_C(u)=C \cap V(G^1)\neq I_C(v)$,
which shows that $C$ separates $u$ and $v$.  If $u\not\in C$, then $u$ is adjacent to a codeword in $G^2$, and this implies that
$C$ separates $u$ and $v$.  If $u \in  \{x_{j,1}^1, x_{j, 2}^1\}$ for some $j$ such that $G_j=A_1$, say $u=x_{j,1}^1$, then
$x_{j, 2}^1$ separates $u$ and $v$.  If $u=x_{j,2k-1}^1$ for  $k \in [t_j]$, then there exists a codeword $d$ such that
$d\in G_j^1$ but $d \not\in I_C(u)$.  If $v$ does not belong to $G_j^1$, then $d$ separates $u$ and $v$.  If $v$ is in
$G_j^1$, the structure of $A_{t_j}$ shows that $C$ separates $u$ and $v$.  A similar argument shows that $C$ separates
$u$ and $v$ when both belong to $G^2$.  Hence $C$ is and ID code for $H \Box K_2$, and it follows that $\gid(H \Box K_2) \le |V(H)|-1=\gid(H)$.
By Theorem~\ref{thm:lowerbound} we now conclude that $\gid(H \Box K_2)=\gid(H)$.  By induction we have shown
 that if $G \in \mathcal{A} \cup (\mathcal{A} \bowtie K_1)$ has order at least $5$, then $\gid(G \Box K_2) = \gid(G)$.

Finally, notice that if we have $G \in \mathcal{A} \cup (\mathcal{A}\bowtie K_1)$ where $G = G_1 \bowtie G_2$ and $G_1, G_2 \not\in \{A_1, A_2\}$, then
\[\gid((G_1 \bowtie G_2) \Box K_2) = |V(G_1)| + |V(G_2)| -1 = \gid(G_1 \Box K_2) + \gid(G_2 \Box K_2) + 1.\]

\end{proof}
The next immediate question is whether or not the graphs given in the statement of Theorem~\ref{thm:lowerclass} are the only graphs which 
satisfy $\gid(G\Box K_2) = \gid(G)$. Unfortunately, there are an infinite number of graphs that are not contained in the class 
$\mathcal{A} \cup (\mathcal{A} \bowtie K_1)$ which satisfy $\gid(G\Box K_2) = \gid(G)$. For example, consider the graph $G$ obtained from 
$A_2 \bowtie A_2 \bowtie A_2$ as follows. Label the vertices of $A_2 = P_4$ as $u, v, x, y$ and let $u_i, v_i, x_i, y_i$ represent the 
vertices of the $i^{th}$ copy of $A_2$ for $i \in [3]$. To obtain $G$, let $w$ represent an additional vertex and add an edge between $w$ 
and $x_3$ and an edge between $w$ and $y_3$. Figure~\ref{fig:example} depicts the graph $G$ without the edges between vertices of $A_i$ and 
$A_j$ when $i \ne j$, $\{i,j\} \subset\{1, 2, 3\}$.
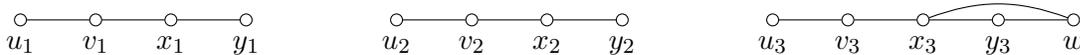
\begin{figure}[h!]
\begin{center}
\begin{tikzpicture}[auto]

	\vertex (1) at (0,0) [scale=.75pt,label=below:$u_1$]{};
	\vertex (2) at (1,0) [scale=.75pt,label=below:$v_1$]{};
	\vertex (3) at (2,0) [scale=.75pt,label=below:$x_1$]{};
	\vertex (4) at (3,0) [scale=.75pt,label=below:$y_1$]{};
	\vertex (5) at (5,0) [scale=.75pt,label=below:$u_2$]{};
	\vertex (6) at (6,0) [scale=.75pt,label=below:$v_2$]{};
	\vertex (7) at (7,0) [scale=.75pt,label=below:$x_2$]{};
	\vertex (8) at (8,0) [scale=.75pt,label=below:$y_2$]{};
	\vertex (9) at (10,0) [scale=.75pt,label=below:$u_3$]{};
	\vertex (10) at (11,0) [scale=.75pt,label=below:$v_3$]{};	
	\vertex (11) at (12,0) [scale=.75pt,label=below:$x_3$]{};
	\vertex (13) at (13,0) [scale=.75pt,label=below:$y_3$]{};
	\vertex (14) at (14,0) [scale=.75pt,label=below:$w$]{};
	\path
		(1) edge (2)
		(2) edge (3)
		(3) edge (4)
		(5) edge (6)
		(6) edge (7)
		(7) edge (8)
                  (9) edge (10)
                  (10) edge (11)
                  (11) edge (13)
                  (11) edge[bend left=20] (14)
                  (13) edge (14)
		;

\end{tikzpicture}
\caption{$G$ obtained from $A_2 \bowtie A_2 \bowtie A_2$}
\label{fig:example}
\end{center}
\end{figure}

We claim that $\gid(G \Box K_2) = 11 = \gid(G)$. First, note that $V(G) - \{u_3, w\}$ is an ID code of $G$. Next, we show that $\gid(G) \ge 11$. 
Let $C$ be a minimum ID code of $G$. If $w \not\in C$, then it is clear that $|C| \ge 11$ since $G[\cup_{i=1}^3 \{u_i, v_i, x_i, y_i\}]$ is isomorphic 
to $A_2 \bowtie A_2 \bowtie A_2$. So assume that $w \in C$. For each $i \in [3]$, $x_i \in C$ in order to separate $u_i$ and $v_i$. Similarly, $v_i \in C$ 
in order to separate $x_i$ and $y_i$. For $i \in [2]$, either $u_i \in C$ or $y_i \in C$ in order to separate $v_i$ and $x_i$ and, with no 
loss of generality, we may assume $u_i \in C$ for $i \in [2]$. Finally, notice that in order to separate $v_1, v_2$, and $v_3$, at least 
two vertices of $\{y_1, y_2, y_3\}$ are in $C$. In any case, we have shown, $\gid(G)\ge 11$. Furthermore, Theorem~\ref{thm:lowerbound} guarantees that 
$\gid(G\Box K_2) \ge 11$. On the other hand, notice that $G\Box K_2$ is illustrated in Figure~\ref{fig:prismexample} and the black vertices form an ID 
code of $G \Box K_2$. Thus, we have constructed a graph $G\not\in \mathcal{A} \cup (\mathcal{A}\bowtie K_1)$ where $\gid(G \Box K_2) = \gid(G)$. Moreover, 
any graph $G$ obtained from the join of $k$ copies of $A_2$ by appending an additional vertex $w$ in the same way as above will satisfy $\gid(G\Box K_2) = \gid(G)$.
\begin{figure}[h!]
\begin{center}
\begin{tikzpicture}[auto]

	\vertex (1) at (0,0) [scale=.75pt,fill=black]{};
	\vertex (2) at (1,0) [scale=.75pt,fill=black]{};
	\vertex (3) at (2,0) [scale=.75pt,fill=black]{};
	\vertex (4) at (3,0) [scale=.75pt]{};
	\vertex (5) at (5,0) [scale=.75pt,fill=black]{};
	\vertex (6) at (6,0) [scale=.75pt]{};
	\vertex (7) at (7,0) [scale=.75pt,fill=black]{};
	\vertex (8) at (8,0) [scale=.75pt]{};
	\vertex (9) at (10,0) [scale=.75pt,fill=black]{};
	\vertex (10) at (11,0) [scale=.75pt]{};	
	\vertex (11) at (12,0) [scale=.75pt,fill=black]{};
	\vertex (13) at (13,0) [scale=.75pt]{};
	\vertex (14) at (14,0) [scale=.75pt]{};
	\vertex (1a) at (0,1) [scale=.75pt]{};
	\vertex (2a) at (1,1) [scale=.75pt]{};
	\vertex (3a) at (2,1) [scale=.75pt]{};
	\vertex (4a) at (3,1) [scale=.75pt]{};
	\vertex (5a) at (5,1) [scale=.75pt]{};
	\vertex (6a) at (6,1) [scale=.75pt,fill=black]{};
	\vertex (7a) at (7,1) [scale=.75pt]{};
	\vertex (8a) at (8,1) [scale=.75pt,fill=black]{};
	\vertex (9a) at (10,1) [scale=.75pt]{};
	\vertex (10a) at (11,1) [scale=.75pt,fill=black]{};	
	\vertex (11a) at (12,1) [scale=.75pt]{};
	\vertex (13a) at (13,1) [scale=.75pt,fill=black]{};
	\vertex (14a) at (14,1) [scale=.75pt]{};
	\path
		(1) edge (2)
		(2) edge (3)
		(3) edge (4)
		(5) edge (6)
		(6) edge (7)
		(7) edge (8)
                  (9) edge (10)
                  (10) edge (11)
                  (11) edge (13)
                  (11) edge[bend left=20] (14)
                  (13) edge (14)
                  (1a) edge (2a)
		(2a) edge (3a)
		(3a) edge (4a)
		(5a) edge (6a)
		(6a) edge (7a)
		(7a) edge (8a)
                  (9a) edge (10a)
                  (10a) edge (11a)
                  (11a) edge (13a)
                  (11a) edge[bend left=20] (14a)
                  (13a) edge (14a)
                  (1) edge (1a)
                  (2) edge (2a)
                  (3) edge (3a)
                  (4) edge (4a)
                  (5) edge (5a)
                  (6) edge (6a)
                  (7) edge (7a)
                  (8) edge (8a)
                  (9) edge (9a)
                  (10) edge (10a)
                  (11) edge (11a)
                  (13) edge (13a)
                  (14) edge (14a)
		;

\end{tikzpicture}
\caption{ID code of $G \Box K_2$}
\label{fig:prismexample}
\end{center}
\end{figure}
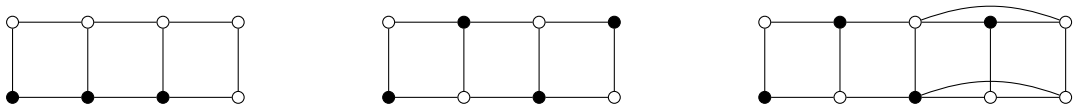

\section{Grid graphs and general upper bounds}\label{grids}
We now give upper bounds for the ID code number of $G \Box P_m$ where $G$ is any graph and $m\ge 2$. First, we consider when $G$ is a path.
\begin{theorem}\label{thm:gridupper}
For any positive integers $m$ and $k$ where $m \le 3k$,
\[\gid(P_m \Box P_{3k}) \le mk + k\left\lceil \frac{m}{3}\right\rceil,\]
\[\gid(P_m \Box P_{3k+1}) \le mk + k \left\lceil\frac{m}{3}\right\rceil + \left\lceil \frac{m}{2}\right\rceil,\]
\[\gid(P_m \Box P_{3k+2}) \le m(k+1) + (k-1)\left\lceil \frac{m}{3}\right\rceil.\]
\end{theorem}

\begin{proof}
First, suppose $m \not\equiv 1 \pmod{3}$. We construct ID codes for each of the above cases. Let $\{0, 1, \dots, m-1\}$ represent the vertices of $P_m$ 
and let $\{0, 1, \dots, y\}$ represent the vertices of $P_{3k+a}$ for $a \in \{0, 1, 2\}$. Define
\[A = \{(i,j) : 0 \le i \le m-1, j \equiv 1 \pmod{3}\}\]
and
\[B = \{(i,j) : i \equiv 1 \pmod{3},  j \equiv 2 \pmod{3}\}.\]

Figure~\ref{fig:0mod3} (a) depicts the set $A \cup B$ for $k=2$ and $a=0$. One can easily verify that $A \cup B$ is an ID code for $P_m \Box P_{3k}$. Next, consider 
$P_m \Box P_{3k+1}$. Let
\[C = \{(i,j) : i \equiv 1 \pmod{2}, j = y\}.\]
Figure~\ref{fig:0mod3} (b) depicts the set $A \cup B\cup C$ for $k=2$ and $a=1$.  Again, it is straightforward to verify that $A \cup B \cup C$ is an ID code for $P_m \Box P_{3k+1}$.
Finally, consider $P_m \Box P_{3k+2}$. Define
\[X = \{(i,j) : 0 \le i \le m-1, j \equiv 1 \pmod{3}, j \ne y\},\]
\[Y = \{(i,j) : i \equiv 1 \pmod{3}, j \equiv 2 \pmod{3}, j \ne y-2\},\]
and
\[Z = \{(i,j) : 0 \le i \le m-1, j = y-1\}.\]

The set $X \cup Y\cup Z$ is an ID code for $P_m \Box P_{3k+2}$.

Now, suppose that $m \equiv 1 \pmod{3}$ and let
\[B' = B \cup \{(i,j) : i = m-1, j \equiv 2\pmod{3}\}\]
and
\[Y' =  Y \cup \{(i,j) : i = m-1, j\equiv 2\pmod{3}, j \ne y-2\}.\]
One can easily verify that $A \cup B'$ is an ID code for $P_m \Box P_{3k}$, $A \cup B' \cup C$ is an ID code of $P_m \Box P_{3k+1}$, and $X \cup Y' \cup Z$ 
is an ID code of $P_m \Box P_{3k+2}$.
\end{proof}
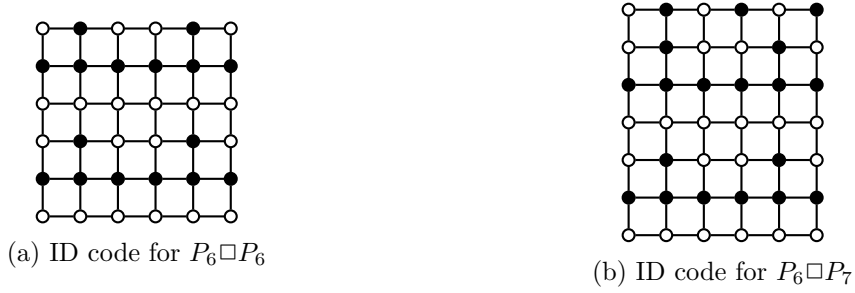
\begin{figure}[ht!]
\begin{subfigure}[]{.5\textwidth}
\centering
\begin{tikzpicture}[scale=1.0,style=thick]
	\vertex (00) at (0,0) [scale=.75pt]{};
	\vertex (01) at (0,0.5) [scale=.75pt,fill=black]{};
	\vertex (02) at (0,1) [scale=.75pt]{};
	\vertex (03) at (0,1.5) [scale=.75pt]{};
	\vertex (04) at (0,2) [ scale=.75pt,fill=black]{};
	\vertex (05) at (0,2.5) [scale=.75pt]{};
	\vertex (10) at (0.5, 0) [scale=.75pt]{};
	\vertex (11) at (0.5, .5) [scale=.75pt,fill=black]{};
	\vertex (12) at (0.5,1) [scale=.75pt,fill=black]{};
	\vertex (13) at (0.5,1.5) [scale=.75pt]{};
	\vertex (14) at (0.5,2)[scale=.75pt,fill=black]{};
	\vertex (15) at (0.5,2.5)[scale=.75pt,fill=black]{};
	\vertex (20) at (1,0)[scale=.75pt]{};
	
	\vertex (21) at (1,.5) [scale=.75pt,fill=black]{};
	\vertex (22) at (1,1) [scale=.75pt]{};
	\vertex (23) at (1,1.5) [scale=.75pt]{};
	\vertex (24) at (1,2) [scale=.75pt,fill=black]{};
	\vertex (25) at (1,2.5) [scale=.75pt]{};
	\vertex (30) at (1.5,0) [scale=.75pt]{};
	\vertex (31) at (1.5,.5) [scale=.75pt,fill=black]{};
	\vertex (32) at (1.5,1)[scale=.75pt]{};
	\vertex (33) at (1.5,1.5)[scale=.75pt]{};
	\vertex (34) at (1.5,2)[scale=.75pt,fill=black]{};
	\vertex (35) at (1.5,2.5)[scale=.75pt]{};
	\vertex (40) at (2,0)[scale=.75pt]{};
	\vertex (41) at (2,.5)[scale=.75pt,fill=black]{};
	\vertex (42) at (2,1)[scale=.75pt,fill=black]{};
	\vertex (43) at (2,1.5)[scale=.75pt]{};
	\vertex (44) at (2,2)[scale=.75pt,fill=black]{};
	\vertex (45) at (2,2.5) [scale=.75pt,fill=black]{};
	\vertex (50) at (2.5,0) [scale=.75pt]{};
	\vertex (51) at (2.5,.5) [scale=.75pt,fill=black]{};
	\vertex (52) at (2.5,1)[scale=.75pt]{};
	\vertex (53) at (2.5,1.5)[scale=.75pt]{};
	\vertex (54) at (2.5,2)[scale=.75pt,fill=black]{};
	\vertex (55) at (2.5,2.5)[scale=.75pt]{};

	\path
		(00) edge (01)
		(01) edge (02)
		(02) edge (03)
		(03) edge (04)
		(04) edge (05)
		(00) edge (10)
		(01) edge (11)
		(02) edge (12)
		(03) edge (13)
		(04) edge (14)
		(05) edge (15)
		(10) edge (11)
		(11) edge (12)
		(12) edge (13)
		(13) edge (14)
		(14) edge (15)
		(10) edge (20)
		(11) edge (21)
		(12) edge (22)
		(13) edge (23)
		(14) edge (24)
		(15) edge (25)
		(20) edge (21)
		(21) edge (22)
		(22) edge (23)
		(23) edge (24)
		(24) edge (25)
		(20) edge (30)
		(21) edge (31)
		(22) edge (32)
		(24) edge (34)
		(23) edge (33)
		(25) edge (35)
		(30) edge (31)
		(31) edge (32)
		(32) edge (33)
		(33) edge (34)
		(34) edge (35)
		(30) edge (40)
		(31) edge (41)
		(32) edge (42)
		(33) edge (43)
		(34) edge (44)
		(35) edge (45)
		(40) edge (41)
		(41) edge (42)
		(42) edge (43)
		(43) edge (44)
		(44) edge (45)
		(40) edge (50)
		(41) edge (51)
		(42) edge (52)
		(43) edge (53)
		(44) edge (54)
		(45) edge (55)
		(50) edge (51)
		(51) edge (52)
		(52) edge (53)
		(53) edge (54)
		(54) edge (55);
\end{tikzpicture}
\caption{ID code for $P_6 \Box P_6$}
\end{subfigure}~
\begin{subfigure}[]{.5\textwidth}
\centering
\begin{tikzpicture}[scale=1.0,style=thick]
	\vertex (00) at (0,0) [scale=.75pt]{};
	\vertex (01) at (0,0.5) [scale=.75pt,fill=black]{};
	\vertex (02) at (0,1) [scale=.75pt]{};
	\vertex (03) at (0,1.5) [scale=.75pt]{};
	\vertex (04) at (0,2) [ scale=.75pt,fill=black]{};
	\vertex (05) at (0,2.5) [scale=.75pt]{};
	\vertex (06) at (0,3) [scale=.75pt]{};
	\vertex (10) at (0.5, 0) [scale=.75pt]{};
	\vertex (11) at (0.5, .5) [scale=.75pt,fill=black]{};
	\vertex (12) at (0.5,1) [scale=.75pt,fill=black]{};
	\vertex (13) at (0.5,1.5) [scale=.75pt]{};
	\vertex (14) at (0.5,2)[scale=.75pt,fill=black]{};
	\vertex (15) at (0.5,2.5)[scale=.75pt,fill=black]{};
	\vertex (16) at (0.5,3) [scale=.75pt,fill=black]{};
	\vertex (20) at (1,0)[scale=.75pt]{};
	
	\vertex (21) at (1,.5) [scale=.75pt,fill=black]{};
	\vertex (22) at (1,1) [scale=.75pt]{};
	\vertex (23) at (1,1.5) [scale=.75pt]{};
	\vertex (24) at (1,2) [scale=.75pt,fill=black]{};
	\vertex (25) at (1,2.5) [scale=.75pt]{};
	\vertex (26) at (1,3) [scale=.75pt]{};
	\vertex (30) at (1.5,0) [scale=.75pt]{};
	\vertex (31) at (1.5,.5) [scale=.75pt,fill=black]{};
	\vertex (32) at (1.5,1)[scale=.75pt]{};
	\vertex (33) at (1.5,1.5)[scale=.75pt]{};
	\vertex (34) at (1.5,2)[scale=.75pt,fill=black]{};
	\vertex (35) at (1.5,2.5)[scale=.75pt]{};
	\vertex (36) at (1.5,3) [scale=.75pt,fill=black]{};
	\vertex (40) at (2,0)[scale=.75pt]{};
	\vertex (41) at (2,.5)[scale=.75pt,fill=black]{};
	\vertex (42) at (2,1)[scale=.75pt,fill=black]{};
	\vertex (43) at (2,1.5)[scale=.75pt]{};
	\vertex (44) at (2,2)[scale=.75pt,fill=black]{};
	\vertex (45) at (2,2.5) [scale=.75pt,fill=black]{};
	\vertex (46) at (2,3) [scale=.75pt]{};
	\vertex (50) at (2.5,0) [scale=.75pt]{};
	\vertex (51) at (2.5,.5) [scale=.75pt,fill=black]{};
	\vertex (52) at (2.5,1)[scale=.75pt]{};
	\vertex (53) at (2.5,1.5)[scale=.75pt]{};
	\vertex (54) at (2.5,2)[scale=.75pt,fill=black]{};
	\vertex (55) at (2.5,2.5)[scale=.75pt]{};
	\vertex (56) at (2.5,3)[scale=.75pt,fill=black]{};

	\path
		(00) edge (01)
		(01) edge (02)
		(02) edge (03)
		(03) edge (04)
		(04) edge (05)
		(00) edge (10)
		(01) edge (11)
		(02) edge (12)
		(03) edge (13)
		(04) edge (14)
		(05) edge (15)
		(10) edge (11)
		(11) edge (12)
		(12) edge (13)
		(13) edge (14)
		(14) edge (15)
		(10) edge (20)
		(11) edge (21)
		(12) edge (22)
		(13) edge (23)
		(14) edge (24)
		(15) edge (25)
		(20) edge (21)
		(21) edge (22)
		(22) edge (23)
		(23) edge (24)
		(24) edge (25)
		(20) edge (30)
		(21) edge (31)
		(22) edge (32)
		(24) edge (34)
		(23) edge (33)
		(25) edge (35)
		(30) edge (31)
		(31) edge (32)
		(32) edge (33)
		(33) edge (34)
		(34) edge (35)
		(30) edge (40)
		(31) edge (41)
		(32) edge (42)
		(33) edge (43)
		(34) edge (44)
		(35) edge (45)
		(40) edge (41)
		(41) edge (42)
		(42) edge (43)
		(43) edge (44)
		(44) edge (45)
		(40) edge (50)
		(41) edge (51)
		(42) edge (52)
		(43) edge (53)
		(44) edge (54)
		(45) edge (55)
		(50) edge (51)
		(51) edge (52)
		(52) edge (53)
		(53) edge (54)
		(54) edge (55)
		(05) edge (06)
		(15) edge (16)
		(25) edge (26)
		(35) edge (36)
		(45) edge (46)
		(55) edge (56)
		(06) edge (16)
		(16) edge (26)
		(26) edge (36)
		(36) edge (46)
		(46) edge (56);
\end{tikzpicture}
\caption{ID code for $P_6 \Box P_7$}
\end{subfigure}
\caption{Examples of ID codes when $m\not\equiv 1\pmod{3}$}
\label{fig:0mod3}
\end{figure}

\begin{theorem}\label{thm:gridlower}
For positive integers $m$ and $n$ where $2 \le m \le n$, $\gid(P_m \Box P_n) \ge mn/3$.
\end{theorem}

\begin{proof} Let $C$ be a minimum ID code of $G = P_m \Box P_n$. Partition $V(G)$ as follows. Let

\[C_1 = \{v \in V(G) : v \in C, v \text{ isolated in } G[C]\},\]
\[C_2 = C - C_1,\]
and for each $i \in [4]$,
\[N_i = \{v \in V(G)- C : v \text{ is adjacent to $i$ vertices in $C$}\}.\]

We further partition $C_1$ and $C_2$ as follows. For $i \in[3]$, let $A_i = \{v \in C_1 : \deg(v) = i+1\}$ and let 
$B_i = \{v \in C_2 : \deg(v) = i+1\}$. Note that the number of edges between $C$ and $V(G)-C$ is at most $2|A_1| + 3|A_2| + 4|A_3| + |B_1| + 2|B_2| + 3|B_3|$.
On the other hand, the number of edges between $C$ and $V(G) - C$ is precisely $|N_1| + 2|N_2| + 3|N_3| + 4|N_4|$. Thus,
\begin{eqnarray*}
|C| + |A_1| + 2|A_2| + 3|A_3| + |B_2| + 2|B_3|&=& 2|A_1| + 3|A_2| + 4|A_3| + |B_1| + 2|B_2| + 3|B_3|\\
&\ge& |N_1| + 2|N_2| + 3|N_3| + 4|N_4|\\
&=& mn - |C| + |N_2| + 2|N_3| + 3|N_4|.
\end{eqnarray*}
Therefore,
\[2|C| + |A_1| + 2|A_2| + 3|A_3| + |B_2| + 2|B_3| \ge mn + |N_2| + 2|N_3| + 3|N_4|.\]
Next, notice that
\[|A_1| + 2|A_2| + 3|A_3| + |B_2| + 2|B_3| \le 3|C_1| + 2|C_2| = 3|C|-|C_2|,\]
which implies that
\[2|C| + 3|C| - |C_2| \ge mn + |N_2| + 2|N_3| + 3|N_4|.\]

On the other hand, no vertex of $N_1$ is adjacent to a vertex of $C_1$ for otherwise $C$ would not separate such a pair of vertices. 
Thus, each vertex of $N_1$ is adjacent to precisely one vertex of $C_2$. Moreover, there can exist no more than $|C_2|$ vertices in $N_1$. 
Therefore, we may conclude that
\begin{eqnarray*}
5|C| &\ge& mn + |N_2| + 2|N_3| + 3|N_4| + |C_2| \\
&\ge& mn + |N_2| + 2|N_3| + 3|N_4| + |N_1|\\
&\ge& mn +mn - |C| + |N_3| + 2|N_4|\\
&\ge& 2mn - |C|.
\end{eqnarray*}
It follows that $|C| \ge mn/3$.
\end{proof}

Note that when $n = 3k$ for some $k \in \mathbb{N}$, it follows from  Theorem~\ref{thm:gridlower}  that $\gid(P_m \Box P_{3k}) \ge mk$. Thus the 
gap between Theorem~\ref{thm:gridupper} and Theorem~\ref{thm:gridlower} is $mk/3$. Now we provide a general upper bound for $\gid(G\Box P_m)$ whenever $m\ge 3$ and $G$ is twin-free.

\begin{theorem}\label{thm:genupper}
For any twin-free graph $G$ of order $n$ and any positive integer $m\ge3$,
\[\gid(G\Box P_m) \le \min\{m\gid(G), m\gamma(G) + \left\lceil \frac{m}{3}\right\rceil(n - \gamma(G))\}.\]
\end{theorem}

\begin{proof}
Let $D$ be an ID code of $G$. Certainly $C = \{(u,v) \mid u \in D, v \in P_m\}$ is an ID code of $G \Box P_m$. Next, let $A$ be a minimum 
dominating set of $G$. Let $\{0, 1, \dots, m-1\}$ represent the vertices of $P_m$. Let $X = \{(u,v) \mid u \in A, v \in P_m\}$ and 
$Y = \{(u,v) \mid u \in V(G) - A, v \equiv 1 \pmod{3}\}$. If $m \not\equiv 1 \pmod{3}$, then $X \cup Y$ is an ID-code of $G \Box P_m$. 
If $m \equiv 1 \pmod{3}$, then let $Y' = \{(u,v) \mid u \in V(G) - A, v \equiv 1 \pmod{3} \text{ or } v = m-1\}.$ The set
$X \cup Y'$ is an ID code of $G \Box P_m$. In either case, we have constructed an ID code of cardinality 
$m\gamma(G) + \left\lceil \frac{m}{3}\right\rceil(n - \gamma(G))$.
\end{proof}

\section*{Acknowledgements}

Research of the first author was supported by a grant from the Simons Foundation (\#209654 to Douglas Rall).

\end{document}